\documentclass{article}

\usepackage{arxiv}

\usepackage[utf8]{inputenc} 
\usepackage[T1]{fontenc}    
\usepackage{hyperref}       
\usepackage{url}            
\usepackage{booktabs}       
\usepackage{amsfonts}       
\usepackage{nicefrac}       
\usepackage{microtype}      
\usepackage{lipsum}		
\usepackage{graphicx}
\usepackage{natbib}
\usepackage{doi}

\usepackage{nicefrac}       
\usepackage{amsthm}
\usepackage{amsmath}
\usepackage{bbm}
\usepackage{comment}
\usepackage{wrapfig}
\usepackage{subcaption}
\usepackage{color}

\graphicspath{{./img/}}

\newcommand{\R}{\mathbb{R}}

\newcommand{\C}{\mathbb{C}}

\newcommand{\e}{\mathrm{e}}
\renewcommand{\d}{\mathrm{d}}
\newcommand{\nn}{\mathrm{NN}}
\newcommand{\numer}{\mathrm{numer}}
\newcommand{\lc}{\ell}

\newtheorem{dfn}{Definition}[section]
\newtheorem{thm}{Theorem}[section]
\newtheorem{rmk}{Remark}[section]
\newtheorem{exm}{Example}[section]

\title{An Error Analysis Framework\\for Neural Network Modeling of Dynamical Systems}

\author{Shunpei Terakawa\\
    Graduate School of System Informatics\\
    Kobe University\\
    Kobe, Hyogo, Japan\\
	\texttt{s-terakawa@stu.kobe-u.ac.jp} \\
	\And
	Takashi Matsubara\\
	Graduate School of Engineering Science\\
	Osaka University\\
	Toyonaka, Osaka, Japan\\
	\texttt{matsubara@sys.es.osaka-u.ac.jp}\\
	\And
	Takaharu Yaguchi\\
	Graduate School of System Informatics\\
    Kobe University\\
    Kobe, Hyogo, Japan\\
	\texttt{yaguchi@pearl.kobe-u.ac.jp} \\
}

\renewcommand{\shorttitle}{}

\begin{document}
\maketitle

\begin{abstract}
	We propose a theoretical framework for investigating a modeling error caused by numerical integration in the learning process of dynamics.
Recently, learning equations of motion to describe dynamics from data using neural networks has been attracting attention.
During such training, numerical integration is used to compare the data with the solution of the neural network model; 
however, discretization errors due to numerical integration prevent the model from being trained correctly.
In this study, we formulate the modeling error using the Dahlquist test equation that is commonly used in the analysis of numerical methods and apply it to some of the Runge--Kutta methods.
\end{abstract}


\section{Introduction}
Data-driven approximation of differential equations by neural networks has a long history~\cite{anderson_comparison_1996,Wang1998RKNN,Oliveira2004,Raissi2019}. 
An important application is learning the governing equations of physical phenomena \cite{Greydanus2019}; for example, in \citet{Greydanus2019}, instead of time-derivatives of the state variables, the energy function is modeled by neural networks, thereby discovering the equation of motion. 
The objective of most of these studies is modeling continuous-time differential equations
$\dot{x}(t) = f(x)$
that describe the target dynamics by using the neural ODE model
    $\dot{x}(t)= f_{\nn} (x)$,
or its extensions.

Due to the difficulty of the observation of the values of $\dot{x}(t)$ in some practical situations, it would be expected that $x(t)$ at enough numbers of $t$'s are observed and hence given as the data.
In this case, numerical integrators (typically, an explicit Runge--Kutta method) are required to integrate the neural network models for learning and also for predicting the dynamics. 
However, the employment of the numerical integrators necessarily induces numerical errors, which results in producing non-negligible \textit{modeling errors} in the learned continuous-time model.
In other words, even if the difference between the data and the numerical solutions of the learned model could be reduced to zero, it does not necessarily mean that the model $f_{\nn}$ matches the target dynamics $f$.

These modeling errors are not a problem if the discrete models can be used as they are; for example, the learned model may be used to predict the dynamics by integrating the model with the same time step as the training data.
In such cases, the modeling error due to the discretization 
has no harmful effect because the model is trained so that 
it can reproduce the data as the numerical solutions completely.
However, the errors can be serious when the models need to be identified 
as continuous differential equations rather than discrete models. This is a common situation, for example, where the target system is a subsystem of a large-scale system. 
In such a case, the subsystems should be identified as continuous ones because each subsystem may have different timescales determined by the data sampling settings and it may not be possible to define a unique time step for the entire system. 
Even if in simpler cases, many commonly used ODE solvers employ the adaptive step size control, so the same kind of problem could happen.

To reduce the modeling errors, the integrators must be replaced or redesigned according to certain criteria.
This paper aims to provide a framework for such an analysis (Figure \ref{fig:overview}). 
In summary, the main contributions of this paper are:
\begin{enumerate}
\item introduction of a framework of theoretical analysis of the modeling errors caused by the numerical integrators, 
\item thereby providing the theoretical background for newly developing integrators that are suitable not only for computation but also for modeling.
\end{enumerate}

\begin{figure}[hbt]
\centering
\quad\quad\includegraphics[width=0.7\linewidth]{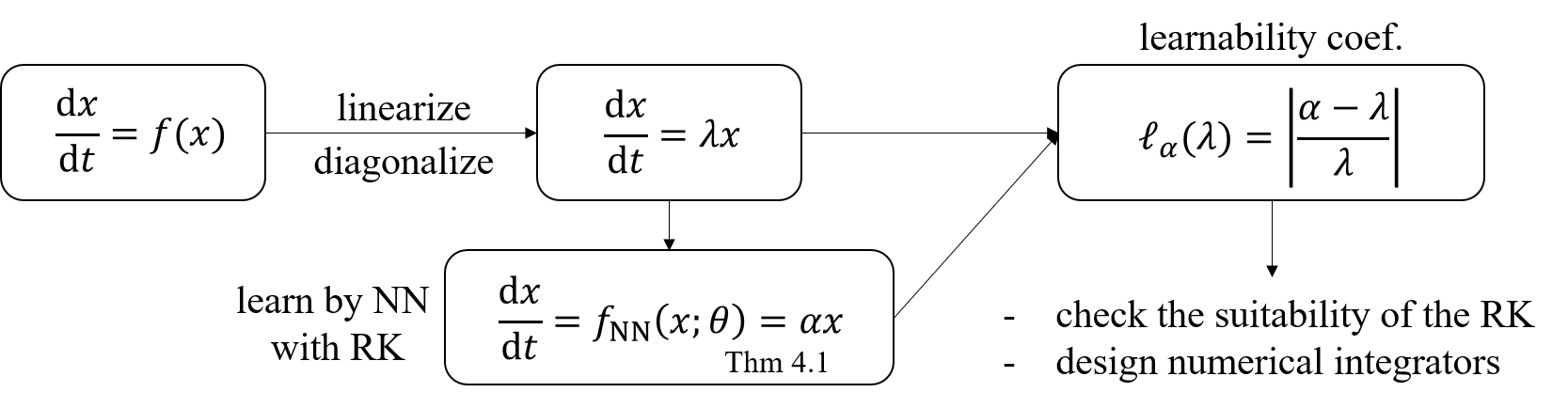}
\caption{
    An overview of the error analysis of the neural network modeling of ODE systems.
    We theoretically estimate the modeling error and propose a framework that enables us to select or design a suitable numerical integrator for learning dynamics according to the characteristics of the target dynamics.
\label{fig:overview}}
\end{figure}

\section{Proposed framework of learnability analysis}
The proposed framework parallels the classical analysis of the stability region. In fact, the problem to be addressed here is to a certain extent similar to the stability analysis of numerical integrators. 
As is well-known, the stability of each numerical integrator depends on the characteristics of the target differential equations, that is, how rapidly the solution decays and/or how rapidly the solution oscillates.
In the stability analysis, the stability region is defined by specifying $\lambda \in \C$ for which the numerical solution to 
the Dahlquist test equation
\begin{align}\label{eq:test-eq}
    \dot{x}(t)
    = \lambda x, \quad x: t \in \R \mapsto x(t) \in \C
\end{align}
by the integrator remains bounded.
The test equation is a representative equation for dynamics in the sense that most nonlinear differential equations describing physical phenomena are reduced to this equation by linearization and diagonalization.
The analysis of the stability region helps users to narrow down their candidates of the integrators for the target differential equations according to the characteristics of the equations. See, e.g., \cite{Hairer2013-fz,Butcher2016} for details.

Following this approach, we propose a framework for analyzing the learnability of numerical integrators. More precisely, we will introduce \textit{the learnability coefficient}, which characterizes the dynamics of which the given numerical integrator is suitable for modeling.



First of all, we confirm the learning process assumed in this paper. We suppose that the target differential equation is learned by the model
    $\dot{\hat{x}}(t) = \hat{f}(\hat{x}; \theta)$,
where $\hat{f}$ is a function that is represented by, e.g., a multilayer perceptron and $\theta$ denotes the model parameters. As a model, we mainly consider neural network models, but we only assume the universal approximation property for the model~(e.g., \citet{Hornik1989}).

For the data, we suppose that only the states $x$ are observable, and therefore the derivatives $\dot{x}$ are not available. To focus on the modeling errors caused by numerical integrators, we consider an ideal situation, where a sufficient amount of the noise-free data are given and they are sampled at a fixed sampling rate $1/h$, thereby supposing that the data are given as a set of pairs $\mathcal{D} := \{ (x_\d^{(n)}, x_\d^{(n+1)})\}$, where $x_\d^{(n)}$ denotes the data sampled at $t=n h$.

For training, 
$\hat{f}(\hat{x}; \theta)$  
is assumed to be learned by minimizing
$
    \sum_{(x_\d^{(n)}, x_\d^{(n+1)}) \in \mathcal{D}} \| {x}^{(n+1)}_\d - \hat{x}^{(n+1)} \|
    $
for a specified norm  $\| \cdot \|$, where $\hat{x}^{(n+1)}$ is given as the numerical solution by the concerned integrator:
$
    \hat{x}^{(n+1)} = {x}_\d^{(n)} + h \hat{f}_{\numer} ({x}_\d^{(n)}, \hat{x}^{(n+1)}; \theta),
    $
where $h \hat{f}_{\numer} ({x}_\d^{(n)}, \hat{x}^{(n+1)})$ is the increment 
numerically computed by the integrator.

Following the stability analysis,
we focus on the case where the target equation is the Dahlquist test equation \eqref{eq:test-eq}. In this case, the data set becomes $\mathcal{D} = \{ (x_\d^{(n)}, x_\d^{(n+1)}= \mathrm{e}^{h \lambda } x_\d^{(n)})\}$ and the loss function is 
\begin{align}
    l(\theta; \mathcal{D}) :=
    \sum\nolimits_{x_\d^{(n)}} \| \mathrm{e}^{h \lambda } x_\d^{(n)} - \hat{x}^{(n+1)} \|.
    \label{eq:test-learn}
\end{align}
As we assumed the universal approximation property of the model, $\hat{f}$ can represent arbitrary functions by appropriately choosing the parameters $\theta$.
Hence in particular $\hat{f}$ can be a linear function $\hat{f}(\hat{x}; \theta) = \alpha \hat{x}$ with $\alpha \in \mathbb{C}$, which is in the same class of functions as the target equation $\dot{x} = \lambda x$.
In fact, there exists an optimal linear function that reduces the loss function to zero (see Theorem \ref{thm:1}.)
By using such $\alpha$, we define the learnability coefficient in the following way.

\begin{dfn}
    For each $\alpha$ that eliminates \eqref{eq:test-learn}, we define the learnability coefficient $\lc_\alpha$ by
    \begin{align*}
        \lc_\alpha := \left| \frac{\alpha - \lambda}{\lambda} \right|.
    \end{align*}
\end{dfn}
We also define the relative error of the real part and the imaginary part of the $\alpha$ independently as the learnability coefficients for decaying and oscillating part.
\begin{dfn}
    For each $\alpha$ that eliminates \eqref{eq:test-learn}, we define the componentwise learnability coefficients $\lc_\alpha^{\mathcal{R}}$ and $\lc_\alpha^{\mathcal{I}}$ for the real part and the imaginary part respectively by 
    $\displaystyle
        \lc_\alpha^{\mathcal{R}} := \left| \frac{\mathrm{Re}\ \alpha - \mathrm{Re}\ \lambda}{\mathrm{Re}\ \lambda} \right|
        \ ,\ \ 
        \lc_\alpha^{\mathcal{I}} := \left| \frac{\mathrm{Im}\ \alpha - \mathrm{Im}\ \lambda}{\mathrm{Im}\ \lambda} \right|.
    $
\end{dfn}

\section{The learnability analysis of the Runge--Kutta methods}

In this section, we show the learnability coefficient for the general Runge--Kutta methods:
\begin{align}
    \hat{x}^{(n+1)} = \hat{x}^{(n)} + h \sum_{i=1}^{p} b_i k_i, \quad
    k_i = \hat{f} (
    \hat{x}^{(n)} + h \sum_{j=1}^p a_{ij} k_j
    ), \label{eq:rk-method}
\end{align}
where $p$, $a_{ij}$'s, and $b_{j}$'s are the constants that define the 
method (see, e.g., \citet{Butcher2016}). 
The matrix and the vector defined by $a_{ij}$ and $b_j$ are respectively denoted by $A$ and $b$.

\begin{thm}[main result]\label{thm:1}
    If the equation 
    \eqref{eq:test-eq}
    is discretized by a Runge--Kutta method \eqref{eq:rk-method}, there exists an $\alpha$ such that the model with
       $ \hat{f}(\hat{x}) = \alpha \hat{x} $
    reduces the loss function \eqref{eq:test-learn} to zero.
    Moreover, $\alpha$ is a solution to
    \begin{align}
         1 +  h \alpha b^{\top}(I - h \alpha A)^{-1} \mathbbm{1} - \e^{h \lambda} = 0, \quad
        \det (I - h \alpha A) \neq 0, \label{eq:alpha}
        \quad \mathbbm{1} = (1 \ 1 \ \cdots \ 1)^\top.
    \end{align}
\end{thm}

\begin{proof}[Proof of Theorem \ref{thm:1}]
    Suppose that the loss function \eqref{eq:test-learn} vanishes for the model with
        $\hat{f}(\hat{x}) = \alpha \hat{x}$.
    If this equation is discretized by the Runge--Kutta method with the initial condition $\hat{x}(n h) = x^{(n)}_\d$, the following equation holds
    \begin{align*}
        \hat{x}^{(n+1)} = x^{(n)}_\d + h b^\top k, \quad
        k = \alpha  ({x}^{(n)}_\d \mathbbm{1} + h A k).
    \end{align*}
    If $\det (I - h \alpha A) \neq 0$ and $\| \hat{x}^{(n+1)} - x_\d^{(n+1)} \|$ is zero for all ${x}^{(n)}_\d$, we get
    $
        \hat{x}^{(n+1)}
        = \mathrm{e}^{h \lambda} {x}^{(n)}_\d
        = x^{(n)}_\d + {x}^{(n)}_\d h \alpha b^\top (I - h \alpha A)^{-1}  \mathbbm{1}.
    $
    Hence, $\alpha$ should satisfy
        $\mathrm{e}^{h \lambda }
        = 1 + h \alpha b^\top (I - h \alpha A)^{-1}  \mathbbm{1}.
        $
\end{proof}

\begin{dfn}
    We call equation \eqref{eq:alpha} \textit{the learnability equation} for the Runge--Kutta method.
\end{dfn}

\begin{rmk}
    In general, equation \eqref{eq:alpha} admits $p$ solutions and hence $p$ learnability coefficients exist for a Runge--Kutta method with $p$ stages. 
    In particular, the model is not uniquely determined when trained as assumed in this paper; see the examples below.
\end{rmk}

\begin{thm}\label{thm:2}
    For the Runge--Kutta methods, the learnability coefficient
    is a function of $z := h \lambda$.
\end{thm}

\begin{proof}[Proof of Theorem \ref{thm:2}]
    From the learnability equation, we have
        $ 1 + (\alpha/\lambda) h \lambda b^{\top}(I - (\alpha/\lambda) h \lambda A)^{-1} \mathbbm{1} - \e^{h \lambda} = 0, 
        \det (I - (\alpha/\lambda) h \lambda A) \neq 0.$
    Therefore, we get
        $
         1 + (\alpha/\lambda) z b^{\top}(I - (\alpha/\lambda) z A)^{-1} \mathbbm{1} - \e^{z} = 0, 
         \det (I - (\alpha/\lambda) z A) \neq 0,
         $
    which shows that $\alpha/\lambda$ and hence ${\alpha}/{\lambda} - 1$ are functions of $z$.
\end{proof}

\begin{thm}\label{thm:3}
    For the Runge--Kutta methods, the componentwise learnability coefficients are functions of $z := h \lambda$ and $\overline{z} := h \overline{\lambda}$.
\end{thm}

\begin{proof}[Proof of Theorem \ref{thm:3}]
    We show the proof for $\lc_\alpha^{\mathcal{R}}$.
    If $\alpha$ is the solution of the learnability equation for $\lambda$,
    then $\overline{\alpha}$ is also the solution of the equation for $\overline{\lambda}$.
    Thus, in the same way as the proof of Theorem \ref{thm:2}, $\overline{\alpha}/\overline{\lambda}$ is shown to be a function of $\overline{z} := h \overline{\lambda}$.
    On the other hand, a part of the definition of $\lc_\alpha^{\mathcal{R}}$ is rearranged in
    $\displaystyle
        \frac{\;\mathrm{Re}\ \alpha\;}{\mathrm{Re}\ \lambda}
        =  \frac{\;\alpha + \overline{\alpha}\;}{\lambda + \overline{\lambda}}
        =  \left.\left(\cfrac{\;\alpha\;}{\lambda} + \cfrac{\;\overline{z}\;}{z} \cfrac{\;\overline{\alpha}\;}{\overline{\lambda}}\ \right)\middle/\left(1 + \cfrac{\;\overline{z}\;}{z}\right)\right.,
    $
    which shows that $\mathrm{Re}\ \alpha/\mathrm{Re}\ \lambda$ and hence $|\mathrm{Re}\ \alpha/\mathrm{Re}\ \lambda - 1|$ is a function of $z$ and $\overline{z}$.
\end{proof}

\begin{rmk}\label{rmk:scaling}
    This scaling property of the learnability coefficients is important as evaluation criteria for designing numerical integrators for learning dynamics. 
    Evaluation criteria should be determined by the modeling errors, while the modeling errors depend on $h$. 
   However, because changing $h$ will change the measured performance of the integrators, naive evaluation criteria that depend on 
    $h$ are not appropriate for designing numerical integrators. 
    Thus evaluation criteria with certain invariance with respect to $h$ are indispensable.
\end{rmk}

\section{Examples} \label{sec:example}
\begin{exm}
    The learnability equation of the explicit Euler method is
        $1 + h \alpha - \e^{h \lambda} = 0$,
    which gives a unique $\alpha$: $\alpha =  {\e^{h \lambda} - 1}/{h}$.
    In addition, the learnability coefficient, which is a relative modeling error, is
    $\ell_\alpha
        =  \left| (\e^{z} - 1)/z - 1 \right|$ with $z = h \lambda$.
\end{exm}

\begin{figure*}[tbh]
    \centering
    \begin{minipage}[t]{0.23\textwidth}
        \centering
        \includegraphics[width=\textwidth]{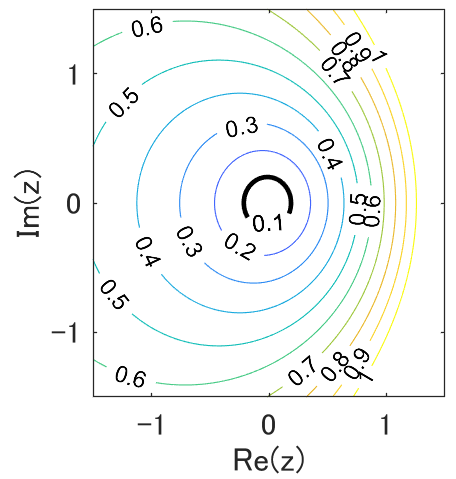}
        \subcaption{Explicit Euler method.}
    \end{minipage}
    \quad
    \begin{minipage}[t]{0.23\textwidth}
        \centering
        \includegraphics[width=\textwidth]{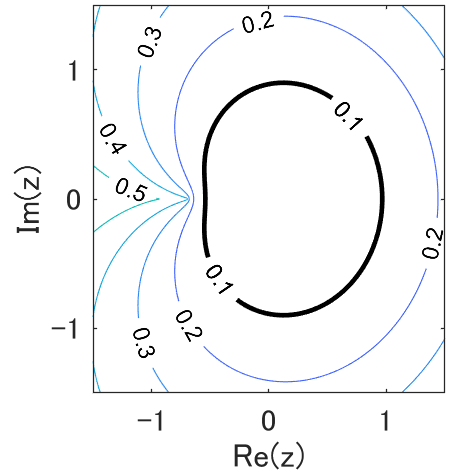}
        \subcaption{Explicit midpoint method ($\alpha_+$). }
    \end{minipage}
    \quad
    \begin{minipage}[t]{0.23\textwidth}
        \centering
        \includegraphics[width=\textwidth]{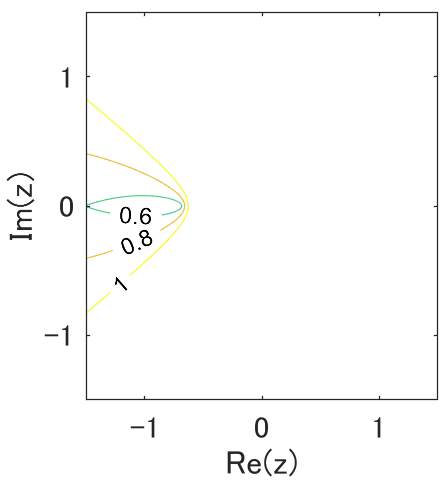}
        \subcaption{Explicit midpoint method ($\alpha_-$).}
    \end{minipage}
    \caption{The contour lines of the learnability coefficients. 
    \label{fig:contour}}
\end{figure*}

\begin{exm}
    For the explicit midpoint method, the learnability equation becomes
        $1 + h \alpha + \frac{h^2 \alpha^2}{2} - \e^{h \lambda} = 0$
    and admits the two solutions:
        $\alpha = (- 1 \pm \sqrt{2 \e^{\lambda h} - 1})/h$
    , where the square root represents the principal value.
    Among these two solutions, calculating the Taylor series expansion around $\lambda = 0$, $\alpha_+ := (- 1 + \sqrt{2 \mathrm{e}^{h \lambda} - 1})/h$ is a 2nd-order approximation to $\lambda$.
    Meanwhile, $\alpha_- := (- 1 - \sqrt{2 \mathrm{e}^{h \lambda} - 1})/h$ is not an approximation. 
    This means that the model with the explicit midpoint method is not uniquely identifiable and, moreover, the learned model may be completely different from the true dynamics.
\end{exm}

The contour lines of the learnability coefficients for the above methods are shown in Figure \ref{fig:contour}. As is expected, the errors are smaller for the explicit midpoint method than for the Euler method when the model corresponding to $\alpha_+$ is learned. Meanwhile, it can be seen from the figure that the midpoint method is not effective for dynamics with strong damping
since the error increases as $\lambda$ goes in the negative direction on the real axis.

\begin{exm} \label{exm:rk4}
    The learnability equation of the classical 4th order Runge--Kutta method is
        $
        1 + h \alpha + \frac{h^2 \alpha^2}{2} + \frac{h^3 \alpha^3}{6} + \frac{h^4 \alpha^4}{24} - \e^{h \lambda} = 0.
        $
    We computed all of the solutions of the equation numerically and extracted the solution 
    closest to $\lambda$ as $\alpha$. 
    The contour lines of the learnability coefficient computed in the above way 
    are shown in Figure \ref{fig:rk4-region}.
\end{exm}
Example \ref{exm:rk4} well explains the result of the learning test of the equation with $\lambda=1.5i, h=1$ shown in Figure \ref{fig:rk4-traj-comparison}.
The experiment was conducted under the following settings.

The neural network model $f_{\nn}$ is a multilayer perceptron, consisting of a fully-connection hidden layer. The input and output layers have 2 units that correspond to the real part and the imaginary part of the input and the output.
The number of hidden units was 200. We used $\tanh$ as the activation function.
We used the Adam method for training and specified the learning rate as 0.001.
Since we used the neural network models, the model function $\hat{f}$'s were not linear, we regarded the average value of $\hat{x}^{(n+1)}/x_\d^{(n)}$ over all $x_\d^{(n)}$'s used in the training process as an estimation of $\alpha$. Actually, for each model, these values were almost constant.  

For the $\lambda$, we prepared the training data as tuples
    $\{(x_1, x_0) \mid x_1 = \e^{h \lambda} x_0\}$
where we uniformly randomly sampled 10000 points for $x_0$'s from $-10 \leq \mathrm{Re}\ x_0 \leq 10, -10 \leq \mathrm{Im}\ x_0 \leq 10$.
After that, the neural network was trained by minimizing the mean squared error
    $M^{-1} \sum_{x_0} \| x_1 - \hat{x}_1\|_2^2,$
where $M$ is the number of the data and $\hat{x}_1$ is the numerical one-step solution of the model $f_{\nn}$ using the classical 4th order Runge--Kutta method.

The black line shows the real part of the exact solution of the target dynamics.
The blue line is the predicted dynamics using $\alpha$ above, and the red one is the actually learned dynamics $f_{\nn}$ in the experiment calculated using \texttt{scipy.integrate.solve\_ivp} with \texttt{RK45} option.
These all oscillate at similar frequencies, but the amplitude of the blue and red ones increases unlike the black one.
The behavior is consistent with the learnability coefficients for the real and imaginary parts.
In fact, $\lc_{\alpha}^{\mathcal{I}}$ is relatively small at $\lambda=1.5$, explaining the small errors in the frequencies; on the other hand, $\lc_{\alpha}^{\mathcal{R}}$ takes large values around the imaginary axis a little away from the origin, which results in the errors in the amplitudes.

\begin{figure*}[tbh]
    \centering
    \includegraphics[width=0.8\textwidth]{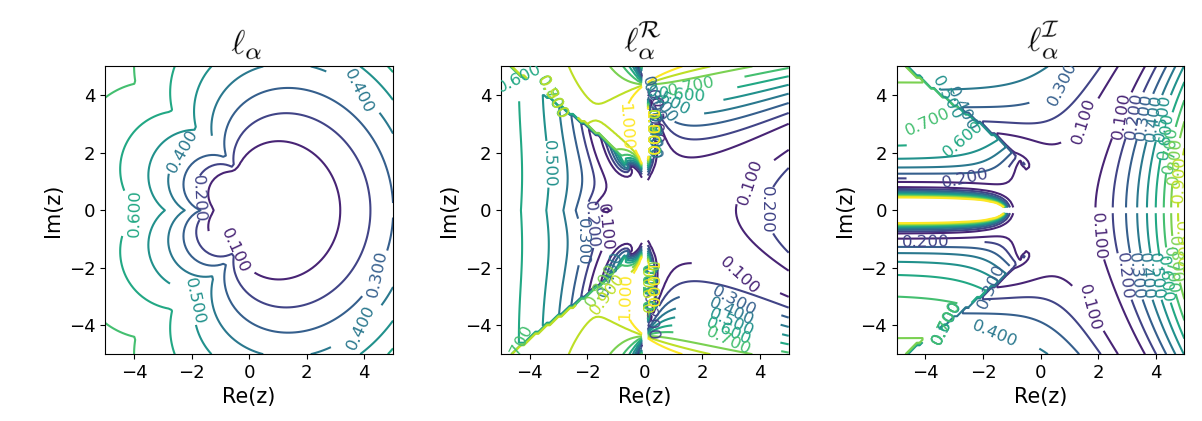}
    \caption{The contour lines of the learnability coefficients of the classical 4th order Runge--Kutta method. The $\alpha$ used to calculate the coefficients is found numerically, unlike the case of Euler method and the explicit midpoint method. 
        \label{fig:rk4-region}}
\end{figure*}

\begin{figure}[tbh]
    \centering
    \begin{minipage}[t]{0.45\linewidth}
        \centering
        \includegraphics[width=1.0\linewidth]{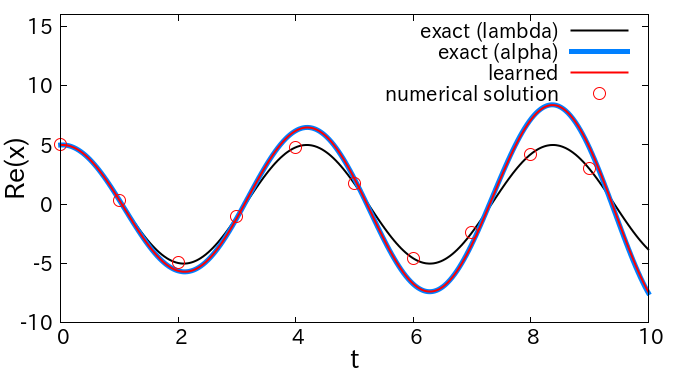}
        \caption{A comparison between the dynamics actually learned and the dynamics predicted by the theoretical results.
            \label{fig:rk4-traj-comparison}}
    \end{minipage}
    \quad
    \begin{minipage}[t]{0.45\linewidth}
        \centering
        \includegraphics[width=0.75\linewidth]{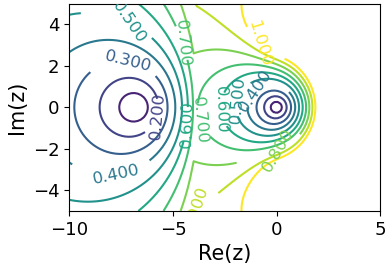}
        \caption{The learnability coefficient of the method determined by \eqref{eq:stab_chebyshev}.}\label{fig:rock}
    \end{minipage}
\end{figure}

\section{Application for designing numerical schemes} \label{sec:design}
A motivation for the introduction of the learnability coefficient is designing numerical integrators for learning differential equations. 
To this end, we propose a procedure for designing integrators.  

First, a region $\Omega$ on which the learning coefficient $l_\alpha(\lambda)$ should be small must be specified 
according to the characteristics of the target phenomena; for example, if we want to model dissipative phenomena, $\Omega$ should contain the negative real axis.

Second, because the learning coefficient $l_\alpha(\lambda)$ must satisfy \eqref{eq:alpha}, instead of specifying the Runge--Kutta method by the Butcher tableau, we design the corresponding equation \eqref{eq:alpha}. In fact, the term 
        $1 + h \alpha b^{\top}(I - h \alpha A)^{-1} \mathbbm{1}$
    in this equation is essentially the stability polynomial of the Runge--Kutta method~\cite{Hairer2013-fz}. 
    Realization methods of Runge--Kutta methods from the given stability polynomial have been developed.
    For example, in the Lebedev method \cite{Lebedev1989-ey, Lebedev1994-yc} the desired numerical integrator is implemented as a composition of a series of simple numerical schemes like the explicit Euler method.

As an illustration, we consider numerical integrators for learning dissipative differential equations; 
we want to design integrators of which 
the learnability coefficient $l_\alpha(\lambda)$ is small on an interval $[-r_0, 0)$ with $r_0 > 0$ as large as possible.

First, we rewrite \eqref{eq:alpha} to
        $1 + \mu z b^{\top}(I - \mu z A)^{-1} \mathbbm{1} - \e^{h \lambda} = 0$, 
     where $\mu = \alpha/\lambda$ and $z = h \lambda$.
     Because $\mu = 1$ is preferable, we want the absolute value of the left-hand side of
        $1 + z b^{\top}(I - z A)^{-1} \mathbbm{1} - e^z = 0$ 
     to be small on $[-r_0, 0)$. If $r_0$ is sufficiently large and hence $e^{z}$ is small, 
     we can consider
        $1 + z b^{\top}(I - z A)^{-1} \mathbbm{1}$
    for simplicity.
   For this function to be small, we set this function to oscillate around zero while satisfying the conditions for ensuring that the associated Butcher tableau certainly defines a numerical integrator. This approach is employed to design a class of highly stable explicit Runge--Kutta methods for dissipative differential equations; in those researches, it is known that      the Chebyshev polynomials are optimal solutions in the sense that large $r_0$ can be used.

 For example, the 2-stage method is given by
    \begin{align}\label{eq:stab_chebyshev}
        1 + b^{\top}(I - z A)^{-1} \mathbbm{1} z
        = 1 + z + \frac{1}{8}z^2.
    \end{align}
    The learnability coefficient of this method is shown in Figure \ref{fig:rock}, in which we can confirm the quite better performance on the negative real axis than the explicit midpoint method, which is also a 2-stage method. For higher stage methods defined by the Chebyshev polynomials, see \cite{Hairer2013-fz}.

\section{Conclusions}

In recent years, methods for constructing differential equation models from data by using deep neural networks have been widely studied.
In such methods, the models {are often discretized by numerical integrators} when learning, but the effects of the discretization have not been well studied theoretically.
To appropriately select and/or design numerical integrators, evaluation criteria for the errors are required. In this paper, we have introduced the learnability coefficient as such a criterion along with the detailed analysis of {Runge--Kutta methods} and designed a method for learning dissipative systems.
As future work, further investigation on the uniqueness of the learned model is needed.

\section*{Acknowledgements}

This work was supported by the JST CREST [Grant Number JPMJCR1914], JST PRESTO [Grant Number JPMJPR21C7] and JSPS KAKENHI [Grant Number 20K11693].

\bibliographystyle{unsrtnat}
\bibliography{references}  






\end{document}